\documentclass[a4paper, 11pt, underruledhead]{paper}
\usepackage{amssymb}
\usepackage{amsmath}
\usepackage{amsthm}
\usepackage{math}
\usepackage[mathscr]{eucal}
\usepackage{enumerate}
\usepackage[all]{xy}
\usepackage{graphics}
\CompileMatrices

\def\GrasSpace(#1,#2){{\bf G}_{{#2}}({#1})}

\def\LineOn(#1,#2){\overline{{#1},{#2}\rule{0em}{1,5ex}}}
\def\inc{\strut\rule{3pt}{0pt}\rule{1pt}{9pt}\rule{3pt}{0pt}\strut}
\def\chains{{\cal C}}
\def\predces{\prec}
\def\succes{\succ}
\def\mid{\mathrel{{-\mkern-15mu:}}}

\def\konftyp(#1,#2,#3,#4){\left( {#1}_{#2}\, {#3}_{#4} \right)}

\def\starof{{\mathrm{S}}}
\def\topof{{\mathrm{T}}}
\def\stars{{\cal S}}
\def\tops{{\cal T}}

\def\HervSpSymb{{\mathscr C}{\chi}}
\def\HervSpace(#1){(\HervSpSymb)_{#1}}

\def\myend{{}\hfill{\small$\bigcirc$}}

\title[The Cox and related configurations]{%
The Cox, Clifford, M{\"o}bius, Miquel, and other related configurations 
and their generalizations}
\author{M. Pra{\.z}mowska, K. Pra{\.z}mowski}
\begin{document}

\newtheorem{reprx}[thm]{Representation}
\newenvironment{repr}{\begin{reprx}\normalfont}{\myend\end{reprx}}
\newtheorem{cnstrx}[thm]{Construction}
\newenvironment{constr}{\begin{cnstrx}\normalfont}{\myend\end{cnstrx}}

\newenvironment{ctext}{%
  \par
  \smallskip
  \centering
}{%
 \par
 \smallskip
 \csname @endpetrue\endcsname
}

\maketitle

\begin{abstract}
  The realizability of countable Cox configurations on Miquelian planes is proved.
  A simple way to determine the automorphisms of the Cox configurations is presented.
\end{abstract}

\begin{flushleft}\small
\end{flushleft}

%%%%%%%%%%%%%%%%%%%%%%%%%%%%%%%%%%%%%%%%%%%%%%%%%%%%%%%%%%%%%%%%%%%%%%%%
%%%%%%%%%%%%%%%%%%%%%%%%%%%%%%%%%%%%%%%%%%%%%%%%%%%%%%%%%%%%%%%%%%%%%%%%
%%%%%%%%%%%%%%%%%%%%%%%%%%%%%%%%%%%%%%%%%%%%%%%%%%%%%%%%%%%%%%%%%%%%%%%% INTRO
%%%%%%%%%%%%%%%%%%%%%%%%%%%%%%%%%%%%%%%%%%%%%%%%%%%%%%%%%%%%%%%%%%%%%%%%
%%%%%%%%%%%%%%%%%%%%%%%%%%%%%%%%%%%%%%%%%%%%%%%%%%%%%%%%%%%%%%%%%%%%%%%%

\section*{Introduction}

The smallest reasonable Cox configuration was invented many eyars ago,
and  it appears under various names, depending on the geometry in which 
it was discovered:
M{\"o}bius configuration of points and planes, Steiner-Miquel configuration
of points, lines, and circles, a suitable biplane; it is also a completion
of the Miquel configuration of points and circles/chains.
In each case a suitable configuration plays its   %%an important 
role in geometry.
And it is truly elegant.
In the first Section of the paper we briefly explain these matters.
Let us mention that even nowadays this configuration is a source of interesting researches
which have applications in physics.
\par
Let us consider this smallest Cox configuration as an abstract, purely combinatorial
object.
It can be, clearly, generalized in various ways depending on the way in which
the source configuration is presented. In this note we follow one of these
ways, where Cox configurations, defined in Section \ref{sec:Cox:def}
in terms of elementary combinatorics of finite sets, are associated
with the hypercube graphs. In what follows we prove the realizability
of Cox configurations on spheres 
(cf. Thm. \ref{thm:reprezent}; more generally: in Miquelian chain spaces,
in particular: on quadrics); we also discuss relationships of our configurations
with generalized Desargues configurations.
\par
Finally, our approach enables us to give an easy way to determine the automorphism
group of a Cox configuration (cf. Cor. \ref{cor:Cox:aut}).

%%%%%%%%%%%%%%%%%%%%%%%%%%%%%%%%%%%%%%%%%%%%%%%%%%%%%%%%%%%%%%%%%%%%%%%%
%%%%%%%%%%%%%%%%%%%%%%%%%%%%%%%%%%%%%%%%%%%%%%%%%%%%%%%%%%%%%%%%%%%%%%%%
%%%%%%%%%%%%%%%%%%%%%%%%%%%%%%%%%%%%%%%%%%%%%%%%%%%%%%%%%%%%%%%%%%%%%%%% BASIC
%%%%%%%%%%%%%%%%%%%%%%%%%%%%%%%%%%%%%%%%%%%%%%%%%%%%%%%%%%%%%%%%%%%%%%%%
%%%%%%%%%%%%%%%%%%%%%%%%%%%%%%%%%%%%%%%%%%%%%%%%%%%%%%%%%%%%%%%%%%%%%%%%

\section{Basic: the Steiner-Miquel configuration}\label{sec:history}

Let us start with a verbal presentation of the classical 
Steiner-Miquel configuration (cf. \cite{ehrmann})\footnote{%
this is {\em our} name, this configuration appears in the literature
under quite different names}:
Consider a Veblen configuration on the Euclidean plane i.e. four lines such that
each two distinct of them have a  common point. 
Then four triangles can be formed
from these lines (taken as the sides of the triangles in question), and four circles
circumscribed on these triangles appear. The circles thus obtained all pass through
one point (the so called {\em Miquel point}). It is one of the classical theorems of
so called {\em Advanced Euclidean Geometry}. 

The resulting configuration is much better seen when we retold the whole story on the
so called inversive plane: the one-point completion of the Euclidean plane
i.e. on a sphere. 
Then the lines of the considered
Veblen configuration are four circles through a fixed point, and the whole configuration
appears fully symmetric. Let us present the Steiner-Miquel configuration in this way.

We have eight circles $A_1,A_2,A_3,A_4, B_1,B_2,B_3,B_4$ and the points
$q_A,q_B,q_{\{i,j\}}$ with $\{i,j\}\in\sub_2(\{ 1,2,3,4 \})$ that satisfy the 
incidences given in Table \ref{tab:hervey}.
Clearly, the matrix defines a 
$\konftyp(8,4,8,4)$-configuration 
${\goth H} = \struct{S,\chains}$
in which 
\begin{enumerate}[(I)]\itemsep-2pt
\item\label{circles:ax1}
  through any pair of points there pass either 0 or 2 blocks, 
\item\label{circles:ax2}
  no three points are in two blocks, 
\item\label{circles:ax3}
  each circle has at least three points,
\item\label{circles:ax4}
a triple of points pairwise on a block is on a block, and
\item\label{circles:ax5}
  a triple of pairwise intersecting blocks have a common point.
\end{enumerate}
(cf. axioms for biplanes and semibiplanes, e.g. in \cite{obrazki}).
The Steiner-Miquel configuration is visualized in Figure \ref{fig:hervey}.
It is self dual (cf. \ref{cor:cx:homog}).

\begin{table}[h!]
\def\OK{\times}
\def\NO{\bf -}
\begin{center}
\begin{tabular}{c||c|c|c|c|c|c|c|c}
$\inc$ & $q_A$ & $q_{\{1,2\}}$ & $q_{\{1,3\}}$ & $q_{\{1,4\}}$ & $q_{\{2,3\}}$ & $q_{\{2,4\}}$ & $q_{\{3,4\}}$  & $q_B$
\\ \hline\hline
$A_1$  & $\OK$ & $\OK$ & $\OK$ & $\OK$ & $\NO$ & $\NO$ & $\NO$  & $\NO$
\\ \hline
$A_2$  & $\OK$ & $\OK$ & $\NO$ & $\NO$ & $\OK$ & $\OK$ & $\NO$  & $\NO$
\\ \hline
$A_3$  & $\OK$ & $\NO$ & $\OK$ & $\NO$ & $\OK$ & $\NO$ & $\OK$  & $\NO$
\\ \hline
$A_4$  & $\OK$ & $\NO$ & $\NO$ & $\OK$ & $\NO$ & $\OK$ & $\OK$  & $\NO$
\\ \hline
$B_1$  & $\NO$ & $\NO$ & $\NO$ & $\NO$ & $\OK$ & $\OK$ & $\OK$  & $\OK$
\\ \hline
$B_2$  & $\NO$ & $\NO$ & $\OK$ & $\OK$ & $\NO$ & $\NO$ & $\OK$  & $\OK$
\\ \hline
$B_3$  & $\NO$ & $\OK$ & $\NO$ & $\OK$ & $\NO$ & $\OK$ & $\NO$  & $\OK$
\\ \hline
$B_4$  & $\NO$ & $\OK$ & $\OK$ & $\NO$ & $\OK$ & $\NO$ & $\NO$  & $\OK$
%% \\ \hline\hline
\end{tabular}
\end{center}
\caption{The incidence matrix of the Steiner-Miquel configuration}
\label{tab:hervey}
\end{table}

\begin{figure}[th]
\begin{center}
  \includegraphics{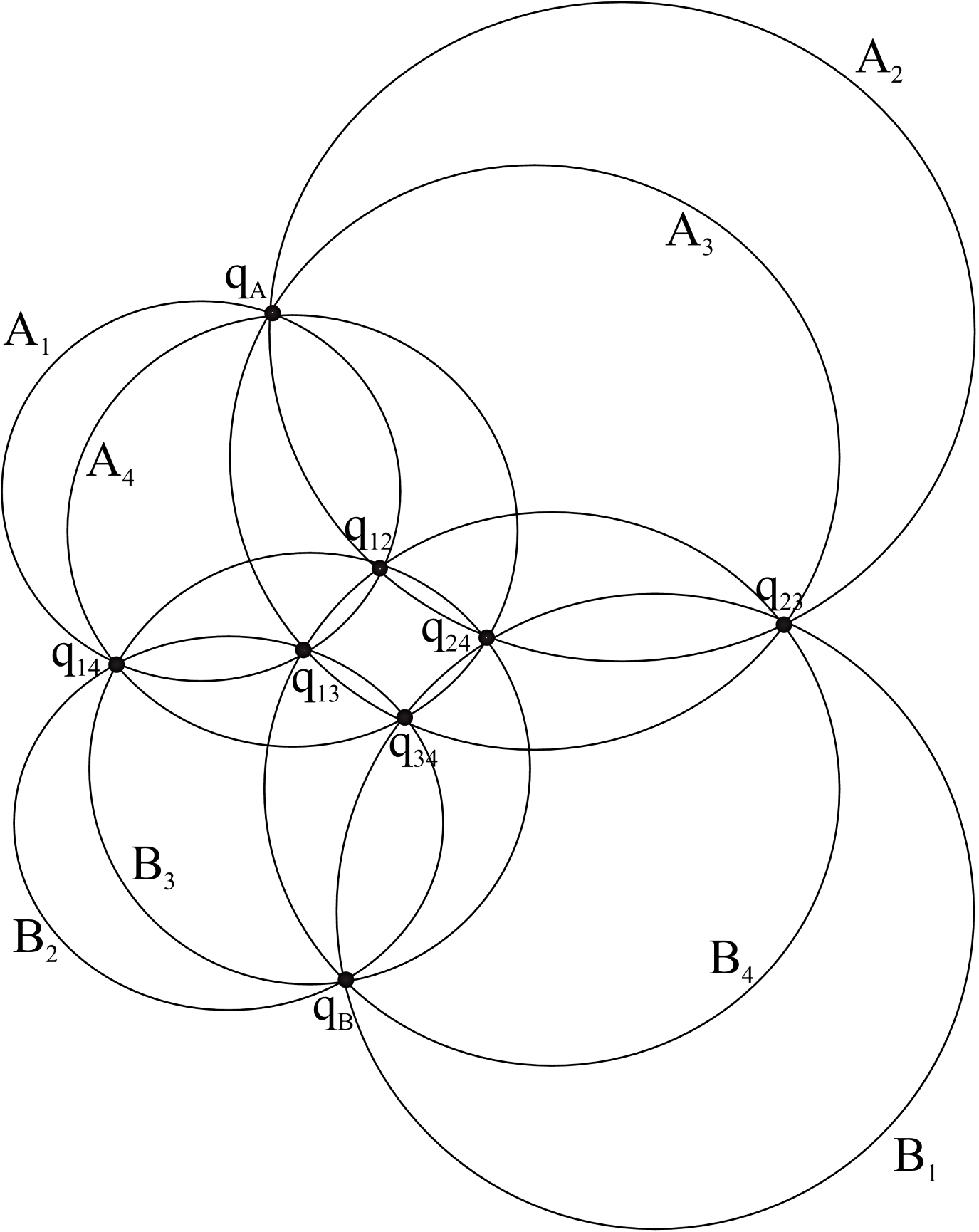}  
\end{center}
\caption{The (classical) Steiner-Miquel configuration}
\label{fig:hervey}
\end{figure}

\begin{fact}
  Let $p\in S$ be a point of $\goth H$. Consider the substructure ${\goth K}_{(p)}$
  of $\goth H$,
  whose blocks are the blocks that pass through $p$ and whose points are the points of 
  rank at least two on thus defined blocks with $p$ excluded.
  Then ${\goth H}_{(p)}$ is the Veblen configuration.
\end{fact}

\begin{fact}
  For each block $D$ of $\goth H$ there is a unique block $D^{\shortmid}$
  such that $D$ and $D^{\shortmid}$ have no point in common.
\end{fact}

\begin{fact}\label{fct:mikinmob}
  Let $D\in\chains$ be arbitrary. The substructure 
  $\struct{{S},{\chains \setminus \{ D,D^{\shortmid} \}}}$ of $\goth H$
  is the Miquel configuration.
\end{fact}

The Miquel Axiom which states, loosely speaking, that each Miquel configuration
closes is considered as a one of fundamental axioms in chain geometry.
It is frequently visualized by the schema in Figure \ref{fig:MiqAx}
\footnote{%
For some modern generalizations to higher dimensions see e.g. \cite{miquel:dim:n}}%
.
%% Various special forms when $a_i=b_i$ is allowed are also discussed in the literature.
The ``thesis'' of this axiom can be rephrased also in the following form,
which is more suitable in ``weak'' chain structures:
{\em if $b_1,b_2,b_3$ are on a chain $B$, then $b_4 \in B$}.

\begin{figure}[th]
\begin{center}
\begin{tabular}{lr}
\begin{minipage}[m]{0.3\textwidth}
\begin{center}
\xymatrix{%
\strut   
& 
{a_1}
 \ar@{-}[rr]
 \ar@{-}[dd]
 \ar@{-}[dl]
&
\strut
&
{a_2}
  \ar@{-}[dd]
  \ar@{-}[dl]
\\
{b_1}
  \ar@{-}[rr]
  \ar@{-}[dd]
&
\strut 
&
{b_2}
  \ar@{-}[dd]
&
\strut
\\
\strut
&
{a_4}
  \ar@{-}[rr]
  \ar@{-}[dl]
&
\strut
&
{a_3}
  \ar@{-}[dl]
\\
{b_4}
  \ar@{-}[rr]
&
\strut
&
{b_3}
&
\strut
} 
% ENDofXYMATRIX
\end{center}
\end{minipage}
&
\begin{minipage}[m]{0.5\textwidth}\it
  if each of the four $4$-sets $a_i,b_i,a_{i+1},b_{i+1}$, $i$ taken mod $4$,
  is on a chain and $a_1,a_2,a_3,a_4$ are on a chain
  (plus: all the points and chains involved are distinct), then 
  $b_1,b_2,b_3,b_4$ are on a chain.
\end{minipage}
\end{tabular}
\end{center}
\caption{The Miquel Axiom and its schema}
\label{fig:MiqAx}
\end{figure}

An incidence structure $\struct{S,\chains}$ 
is called {\em a weak chain structure} when it
satisfies \eqref{circles:ax2}, \eqref{circles:ax3}, and \eqref{circles:ax4}
Then the following is evident
\begin{fact}
  If a weak chain structure 
  satisfies the Miquel Axiom
  then each Steiner-Miquel Configuration contained in it closes.
\end{fact}
Recall also that 
so called Miquelian M{\"o}bius spaces i.e. incidence structures  of points and 
conics on a (projective) sphere (i.e. on a nonruled quadric) in a Pappian
projective space are the classical examples of chain structures 
which satisfy the Miquel axiom.
\par
In view of \ref{fct:mikinmob}, the configuration $\goth H$ contains the dual of 
Miquel Configuration. The so called {\em novel generalization of Clifford's
classical point-circle configuration} (cf. \cite[Def. 3.1]{schief}, \cite{konop} is
exactly a realization of the dual of the Miquel Configuration on the Euclidean
sphere (i.e. in the real M{\"o}bius plane).
At a first look it is intriguing that `novel generalization...' can be defined
as the family of circles circumscribed on the faces of an icosahedron.
It becomes less intriguing when we remind that an icosahedron is considered
as a dual of a cube, and (see Figure \ref{fig:MiqAx}) the Miquel Configuration 
consists of the cicles circumscribed on the faces of a cube.

\medskip
A substructure $\struct{S_0,\chains_0}$ of an incidence structure 
$\struct{S,\chains}$ is called {\em a $(l_1,l_2)$-closed substructure} if 
it meets the following conditions:
\begin{itemize}\itemsep-2pt\def\labelitemi{--\quad}
\item
  $C\in\chains$, $|C \cap S_0|\geq l_1$ imply $C\in\chains_0$, and
\item
  $|\{ C\in\chains_0\colon a \in C \}|\geq l_2$ implies $a\in S_0$.
\end{itemize}
A $(3,1)$-closed substructure of a weak chain structure $\goth M$ is frequently referred
to as a {\em subspace} of $\goth M$.

Finally, the configuration $\goth H$ is also (the combinatorial scheme of) 
the M{\"o}bius $8_4$-configuration (cf. \cite{coxet})
consisting of two tetrahedrons simultaneously inscribed and described %%circumscribed
one in/on the other.

%%%%%%%%%%%%%%%%%%%%%%%%%%%%%%%%%%%%%%%%%%%%%%%%%%%%%%%%%%%%%%%%%%%%%%%%
%%%%%%%%%%%%%%%%%%%%%%%%%%%%%%%%%%%%%%%%%%%%%%%%%%%%%%%%%%%%%%%%%%%%%%%%
%%%%%%%%%%%%%%%%%%%%%%%%%%%%%%%%%%%%%%%%%%%%%%%%%%%%%%%%%%%%%%%%%%%%%%%% COX:DEF
%%%%%%%%%%%%%%%%%%%%%%%%%%%%%%%%%%%%%%%%%%%%%%%%%%%%%%%%%%%%%%%%%%%%%%%%
%%%%%%%%%%%%%%%%%%%%%%%%%%%%%%%%%%%%%%%%%%%%%%%%%%%%%%%%%%%%%%%%%%%%%%%%

\section{The Cox Configuration: definitions}\label{sec:Cox:def}

Let $X$ be a set, $|X|\geq 3$. 
We write $\sub_k(X)$ for the set of $k$-element subsets of $X$, for any integer $k$.
The Cox configuration (comp. \cite{coxet}) is the incidence structure
\begin{equation}\label{eq:cx:def}
  \HervSpace(X) := %% \HervSpace(n) = 
%%  \\
  \bstruct{ {\bigcup \{\sub_{2k}(X)\colon 0\leq k \leq |X| \} },%
  {\bigcup\{\sub_{2k+1}(X)\colon 0\leq k\leq |X| \}},{\predces \cup \succes}},
\end{equation}
where 
$\predces$ is the direct-successor-relation and $\succes \;=\; \predces^{-1}$.
We write $\HervSpace(n) = \HervSpace(X)$ when $n = |X|$.
Clearly, if $n$ is finite then
the Cox configuration $\HervSpace(n)$ is a 
$\konftyp(2^{n-1},n,2^{n-1},n)$-configuration,
and it is a weak chain structure that
meets condition \eqref{circles:ax1}  %%, \eqref{circles:ax2}, and \eqref{circles:ax3}
for arbitrary at least 3-element set~$X$.
\par
We see that $\HervSpace(4) \cong {\goth H}$ (have a look at Figure \ref{fig:herv4}). 
The structure $\HervSpace(3)$ is visualized in Figure \ref{fig:herv3}. 
The structure $\HervSpace(5)$, much more complicated, being a suitable completion of
the Desargues configurations, is presented in Figure \ref{fig:herv5}.

\begin{figure}[t]
\begin{center}
  \includegraphics{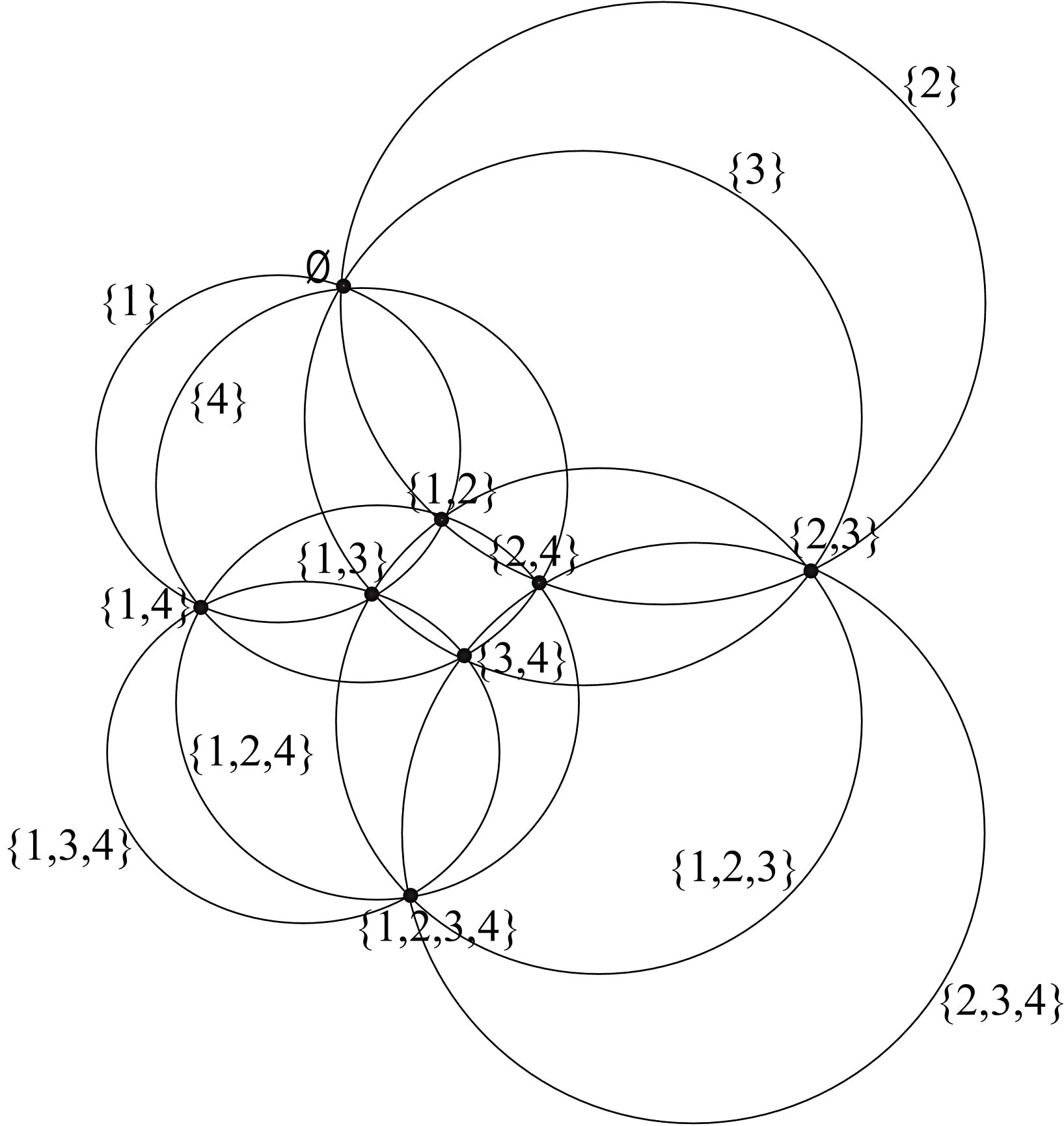}  
\end{center}
\caption{The (classical) Steiner-Miquel configuration 
  labelled by the subsets of $\{ 1,2,3,4 \}$}
\label{fig:herv4}
\end{figure}

\begin{figure}[t]
\begin{center}
 \includegraphics{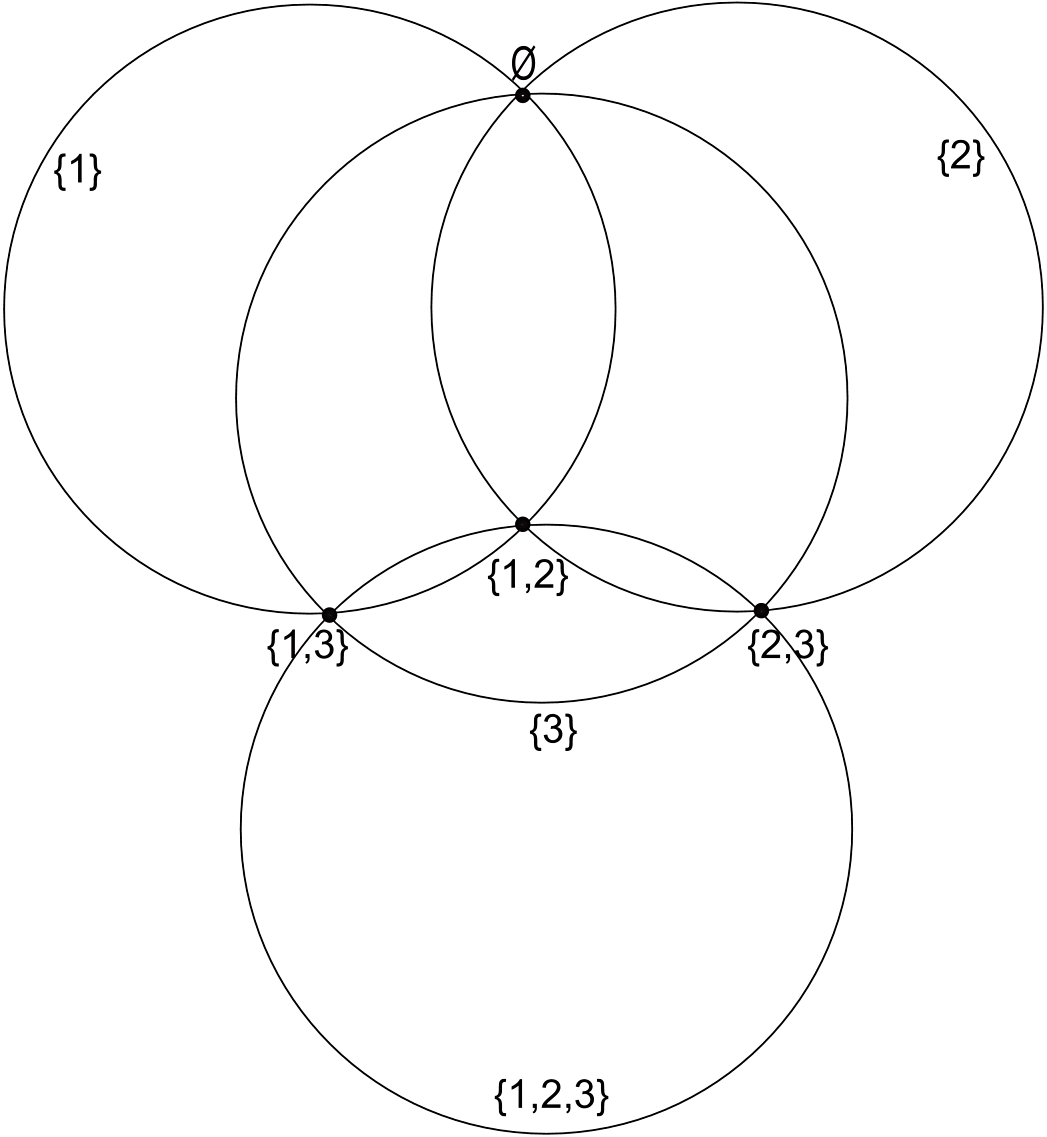}  
\end{center}
\caption{$\HervSpace(3)$: a ``spherical'' triangle}
\label{fig:herv3}
\end{figure}

\begin{figure}[t]
\begin{center}
   \includegraphics{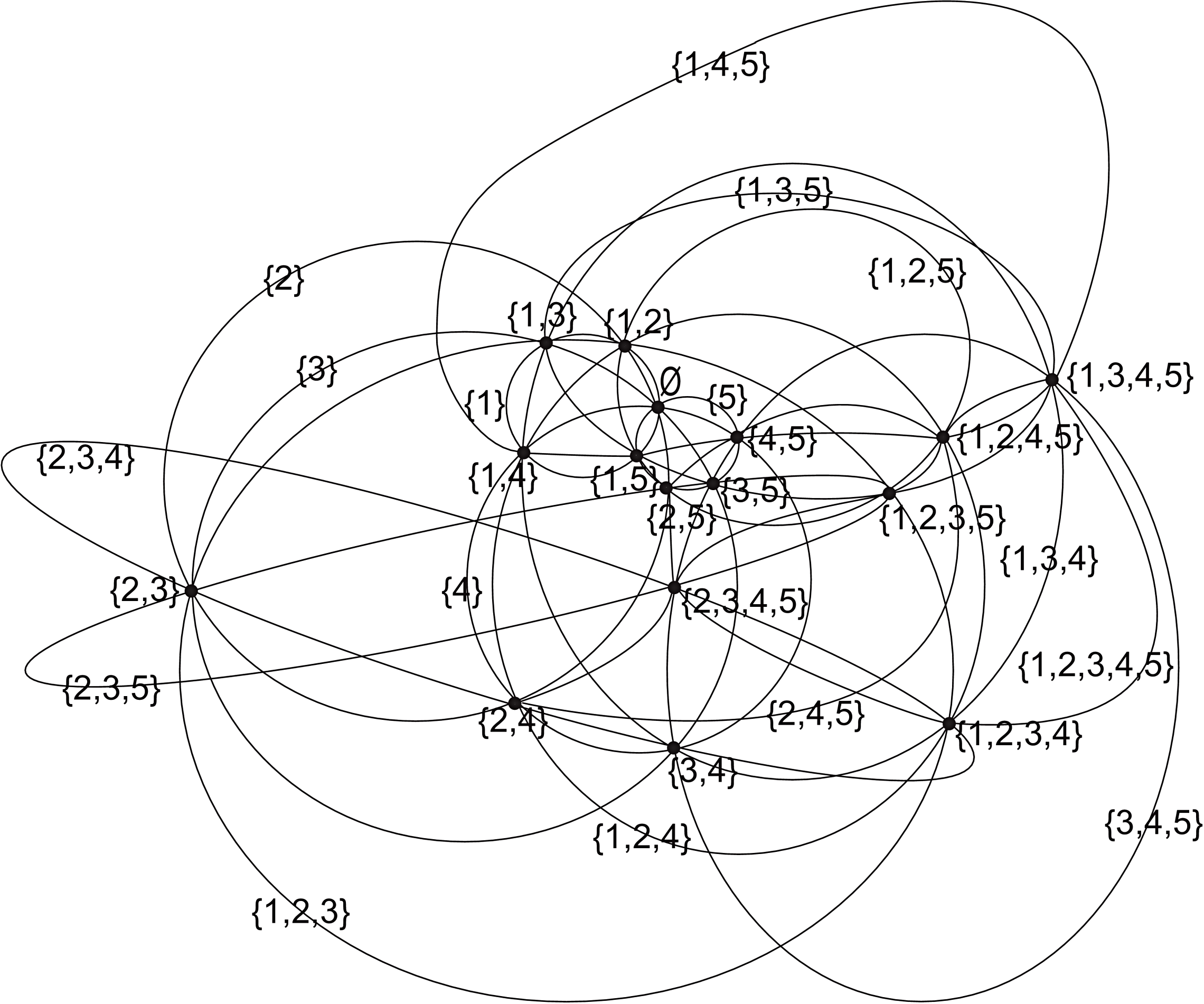}
\end{center}
\caption{$\HervSpace(5)$: the points labelled by the elements of 
$\sub_2(\{ 1,2,3,4,5 \})$ form a Desargues configuration (one can easily check it
using the combinatorial representation; it is fairly not obvious observing the 
figure)}
\label{fig:herv5}
\end{figure}

\begin{exm}\label{exm:gras}
  Let $X$ be a set and $k$ be an integer with $1 < k+1 < |X|$. 
  On the set $\sub_k(X)$ we define two families of its subsets.
  The sets $\topof(b) = \{ z\in\sub_k(X)\colon z \subset b \}$ with 
  $b \in\sub_{k+1}(X)$ are called {\em tops},
  and the sets 
  $\starof(h) = \{ z\in\sub_k(X)\colon h\subset z \}$ with 
  $h\in\sub_{k-1}(X)$ are {\em the stars}.
  Let $\tops_k(X)$ be the set of all the tops and $\stars_k(X)$ be the set of all the stars
  defined on $\sub_k(X)$.
  The incidence structure
\begin{equation}
  \GrasSpace(X,k) := \struct{\sub_k(X),\tops_k(X)}, 
\end{equation}
  called a {\em combinatorial Grassmannian}, is a partial linear space, 
  and for finite $|X|  =:n$ it is a
  $\konftyp({\binom{n}{k}},{n-k},{\binom{n}{k+1}},{k+1})$-configuration 
  (see e.g. \cite{perspect}, \cite{combver}).
  $\GrasSpace(4,2)$ is the Veblen configuration, $\GrasSpace(5,2)$ is the Desargues
  configuration, and from among all the combinatorial Grassmannians the so called
  {\em generalized Desargues configurations} $\GrasSpace(X,2)$ appear
  most important.
  \par
  Write
\begin{equation}
  {\bf K}^\dagger(X,k) := \struct{\sub_k(X),\tops_k(X) \cup \stars_k(X)};
\end{equation}
  The incidence structure ${\bf K}^\dagger(X,k)$ 
  is a weak chain structure which satisfies \eqref{circles:ax1}.
  %% , \eqref{circles:ax2}, and \eqref{circles:ax3}.
  It is 
  \begin{itemize}\def\labelitem{--}\itemsep-2pt
  \item 
    the structure of the maximal cliques of $\GrasSpace(X,k)$, and
  \item
    a (closed) substructure of $\HervSpace(X)$: 
    if $k$ is even, then each block of $\HervSpace(X)$
    that has two points common with the point set of ${\bf K}^\dagger(X,k)$ is actually
    a block of ${\bf K}^\dagger(X,k)$, and dually, when $k$ is odd.
  \end{itemize}
  If $n = 2k$ then
  ${\bf K}^\dagger(X,k)$ is a 
  $\konftyp({\binom{2k}{k}},{2k},{2\binom{2k}{k+1}},{k+1})$-configuration.
%%%
\iffalse
  \par
Finally, 
%
\begin{ctext}
  Let $|X| = 4$ and $k=2$. Then
  ${\bf K}^\dagger(X,k)$ and $\goth H$ with the points $q_A$ and $q_B$ deleted
  are isomorphic.
\end{ctext}
\fi
%%%
\myend
\end{exm}

The following is also evident:
\begin{fact}\label{fct:wstep}
 If $X_1\subset X_2$ then $\HervSpace(X_1)$ is a ($(2,2)-closed$)
 subconfiguration of $\HervSpace(X_2)$.
\end{fact}
\begin{fact}\label{fct:hypcub}
  The Levi graph of the Cox configuration $\HervSpace(X)$ is a hypercube graph.
\end{fact}
Since the hypercube graphs 
(see \cite{bryant}, \cite{coxet}, \cite{obrazki}, and many others)
appear one of most ``classical'' graphs, the above
observation explains why Cox configurations deserve some interest 
also in the context of graph theory.
Let us mention, as an example, one observation:
{\em If $8\leq m\leq n$
and $4 | m$, $8 | n$ or $8 | m$, $4 | n$
then the complete bipartite graph $K_{m,n}$ can be decomposed into
$Q_4$-cubes i.e. $K_{m,n}$ can be presented as the union of the
Levi graphs of the M{\"o}bius $M_8$ configuration}
(a consequence of \cite[Thm. 2.2]{bryant}).
Other connections of this sort are indicated at the end of Section \ref{sec:reprezenty}.

The converse of \ref{fct:hypcub} is slightly more complex.
Let $X$ be arbitrary and $\goth K$ be the hypercube graph
$\struct{\sub(X),\predces \cup \succes}$.
If $X$ is infinite then $\goth K$ is disconnected. Each two its connected components
are isomorphic, and the connected component of $\emptyset$ is 
the set $\sub_{<\omega}(X)$ of the finite subsets of $X$.
And, in turn, the restriction of $\goth K$ to $\sub_{<\omega}(X)$
is the Levi graph of $\HervSpace(X)$.

%%%%%%%%%%%%%%%%%%%%%%%%%%%%%%%%%%%%%%%%%%%%%%%%%%%%%%%%%%%%%%%%%%%%%%%%
%%%%%%%%%%%%%%%%%%%%%%%%%%%%%%%%%%%%%%%%%%%%%%%%%%%%%%%%%%%%%%%%%%%%%%%%
%%%%%%%%%%%%%%%%%%%%%%%%%%%%%%%%%%%%%%%%%%%%%%%%%%%%%%%%%%%%%%%%%%%%%%%% COX:REPR
%%%%%%%%%%%%%%%%%%%%%%%%%%%%%%%%%%%%%%%%%%%%%%%%%%%%%%%%%%%%%%%%%%%%%%%%
%%%%%%%%%%%%%%%%%%%%%%%%%%%%%%%%%%%%%%%%%%%%%%%%%%%%%%%%%%%%%%%%%%%%%%%%

\section{Representations and interpretations}\label{sec:reprezenty}

The (geometrical) procedure which leads to $\goth H$ starting from the Veblen configuration
can be easily generalized as below. Just observe the following facts:
\begin{fact}
  Let $u\in\sub_k(X)$ be a point of $\HervSpace(X)$. The family of blocks of $\HervSpace(X)$
  which pass through $u$ with the point $u$ deleted is the generalized Veblen
  configuration (it consists of $|X|-k$ blocks each two of them having a common point).
\end{fact}
\begin{fact}\label{fct:graskliki}
  Let $k$ be an integer. Note that the elements of $\sub_{k+1}(X)$ are {\em simultaneously
  the lines} of $\GrasSpace(X,k)$ and {\em the points} of $\GrasSpace(X,k+1)$.
  A (maximal) clique $\cal K$ contained in $\sub_{k+1}(X)$ can be either 
  \begin{itemize}\itemsep-2pt
  \item
    the set $\starof(u)$ with $u\in\sub_{k}(X)$: then $\cal K$ consists of the bundle of 
    lines of $\GrasSpace(X,k)$ through a point, or 
  \item
    the set $\topof(U)$ with $U\in\sub_{k+2}(X)$, a line of $\GrasSpace(X,k+1)$: 
    then $\cal K$ can be viewed as the lines of the
    generalized Veblen configuration $\GrasSpace(U,k)$, and geometrically this family
    can be visualized as a {\em plane} contained in $\GrasSpace(X,k)$.
  \end{itemize}
\end{fact}
Fact \ref{fct:graskliki} enables us to present $\HervSpace(X)$ as a suitable completion
of the combinatorial Grassmannian $\GrasSpace(X,k)$.
\begin{repr}\label{repr:gras2cox}
  Observe that, formally, $\GrasSpace(X,0)$ is the structure consisting of the bundle
  of (one-element) lines through a fixed point, and $\GrasSpace(X,|X|-1)$ is the chain
  of points on a single line.
  \par
  Let $2\mathrel{|} k$, $0 < k < |X|$. The points of $\HervSpace(X)$ in $\sub_k(X)$ 
  are exactly the points of $\GrasSpace(X,k)$. 
  The lines of $\GrasSpace(X,k)$ are also blocks of $\HervSpace(X)$.
  With each ``plane" $U$ of $\GrasSpace(X,k)$ we associate a new ``improper" point
  $U^\infty$ and require that all the lines in this plane pass through $U^\infty$.
  These improper points correspond to the points of $\HervSpace(X)$ in $\sub_{k+2}(X)$.
  Clearly, added points are on one block 
  when the planes which determine them have a common line in $\GrasSpace(X,k)$. 
  With any (maximal) family of planes that do not have a line in common
  but pairwise are neighbor (if such a family exists in $\GrasSpace(X,k)$)
  we add a new block which joins their improper points.
  The structure of the improper points and new blocks is (isomorphic to)
  $\GrasSpace(X,k+2)$.
  \par
  Next, with each maximal clique $\cal K$ (a simplex) in $\GrasSpace(X,k)$ we 
  associate a ``clique-block'' ${\cal K}^\circ$ that contains all the points in $\cal K$.
  These clique-blocks correspond to the blocks of $\HervSpace(X)$ in $\sub_{k-1}(X)$.
  As above, with each family $\cal F$ of maximal cliques which do not have a common point but 
  pairwise intersect (if it exists in $\GrasSpace(X,k)$)
  we associate a new ideal point common  to 
  all the clique-blocks ${\cal K}^\circ$ with ${\cal K}\in{\cal F}$.
  And, as above, the structure of these ideal points and clique-blocks is
  $\GrasSpace(X,k-2)$.
  \par
  Continuing in this fashion, finally,
  we end up with $\HervSpace(X)$.
\end{repr}
Presentation \ref{repr:gras2cox} may be especially useful when applied
to a generalized Desargues configuration $\GrasSpace(X,2)$.

\medskip
%% \section{Spherical representation}

Let us stress on that the original Cox configuration (possibly, it would be
better to call it the Steiner-Miquel configuration, as proposed in Section \ref{sec:history}), 
is closely related to the configuration
of the Miquel Axiom (see Figure \ref{fig:MiqAx}), 
which is nowadays called the Miquel Configuration.
And thus it is usually considered as a configuration of circles rather than of planes.
It is mentioned in \cite{coxet} that the configuration $\HervSpace(X)$ can be realized on 
a $2$-sphere for arbitrary set $X$. Justification of this fact seems incomplete, though.
Let us make some comments on the subject.
\begin{prop}\label{prop:cox:miquelian}
  The Cox configuration $\HervSpace(X)$ satisfies the Miquel Axiom
  for every set $X$.
\end{prop}
\begin{proof}
  Let points $a_i,b_i$ satisfy the assumptions of the Miquel Axiom 
  (cf. Figure \ref{fig:MiqAx}). Without loss of generality (cf. \ref{cor:cx:homog})
  we can assume that $a_2 = \emptyset$; then $A = \{ k \}$, $a_1 = \{i,k\}$, 
  $a_3 = \{ j,k \}$, $a_4 = \{ k,l \}$ for some $i,j,k,l\in X$. From the definition of 
  $\HervSpace(X)$ we compute that $b_2 = \{i,j\}$, $b_1 = \{ i,l \}$,
  $b_3 = \{ j,l \}$, and $b_4 = \{ i,j,k,l \}$, which are on the block
  $B = \{ i,j,l \}$.
\end{proof}

\def\naszasfera{{\cal S}}
\begin{prop}\label{prop:cox:planar}
  Let $\HervSpace(X)$ be realized in a weak chain structure $\goth M$ and let
  $\struct{\naszasfera,\chains_0}$ be a subspace of $\goth M$ which contains a pair of 
  blocks of $\HervSpace(X)$ meeting in two distinct points.
  Then $\HervSpace(X)$ is contained in $\naszasfera$.
\end{prop}
\begin{proof}
  Let us refer to the blocks of $\goth M$ as to {\em circles}.
  Without loss of generality we can assume that 
  $\naszasfera$ contains the circles $\{i_1\}$ and $\{i_2\}$
  meeting in $\emptyset$ and $\{i_1,i_2\}$, for some distinct $i_1,i_2\in X$. 
  Each one of the circles
  $\{i\}$, $X\ni i\neq i_1,i_2$ crosses our two circles in three distinct points:
  $\emptyset$, $\{i_1,i\}$, and $\{i_2,i\}$, and thus $\{ i\}$ lies on $\naszasfera$.
  Consequently, $\naszasfera$ contains all the points in $\sub_2(X)$. 
  \par
  Each circle in $\sub_{k+1}(X)$ with $k\geq 2$ contains at least three points in
  $\sub_k(X)$ so, if $\sub_k(X)\subset\naszasfera$, then each circle in $\sub_{k+1}(X)$
  lies on $\naszasfera$. 
  Each point in $\sub_{k+1}(X)$ lies on a circle in $\sub_k(X)$, so if the 
  circles in $\sub_k(X)$ all are subsets of $\naszasfera$, we have 
  $\sub_{k+1}(X)\subset\naszasfera$.
  This proves, inductively, that each point and each circle in $\sub(X)$
  lies on $\naszasfera$.
\end{proof}

In an arbitrary M{\"o}bius space $\goth M$ a pair of chains intersecting in two points
determines a plane of $\goth M$ which contains them, and the same 
holds, generally, in chain spaces associated with quadrics (Minkowski
spaces, Laguerre, etc.). Thus from \ref{prop:cox:planar} we directly derive that
the configuration $\HervSpace(X)$ is {\em planar}:
\begin{cor}\label{cor:cox:planar}
  Assume that the configuration $\HervSpace(X)$ is realized in a 
  chain space $\goth M$ associated with a quadric  %%, or in a generalized Laguerre space
  (as an important particular case: let $\goth M$ be a  M{\"o}bius space).
  Then actually it is realized on a plane in $\goth M$.
\end{cor}

So, without loss of generality we can search for a realization of $\HervSpace(X)$ 
on a 2-sphere.
\begin{thm}[M{\"o}bius Representation]\label{thm:reprezent}
  Let $X$ be a finite (at least $3$-element) or countable set.
  The structure $\HervSpace(X)$ can be realized on a 2-sphere
  $\naszasfera$
  in a Euclidean 3-space, with the blocks interpreted as circles on $\naszasfera$.
\end{thm}
\begin{proof}
  Set $n= |X|$.
  In view of natural connections of $\HervSpace(4)$ with the Miquel configuration,
  the claim is evident for $n=3,4$. 
  We prove that each realization of $\HervSpace({\{ 1,\ldots,n-1 \}})$
  on a 2-sphere in a Euclidean 3-space can be extended to a realization
  of $\HervSpace({\{ 1,\ldots,n \}})$ (comp. \ref{fct:wstep}).
  \par\medskip
  Set $Z = \{ 1,\ldots,n-1 \}$, $X = Z \cup \{ n \}$ and assume that the configuration
  $\HervSpace(Z)$ is already realized on a sphere $\naszasfera$, as required.
  In what follows we shall refer to the subsets of $Z$ 
  {\em as to points and circles (resp.) of }$\naszasfera$, 
  thus identifying a point/block   %%  $a\subset Z$ 
  of the configuration $\HervSpace(Z)$ 
  and its image under a realization:
  a ``real'', ``geometrical'' point/circle.
  We are going to construct a realization of $\HervSpace(X)$ on $\naszasfera$
  extending the given realization of $\HervSpace(Z)$.
  \par
  Let us choose any circle through $\emptyset$ with the following properties:
  it lies on $\naszasfera$, is distinct from $a$,  and is not tangent to $a$
  for each circle $a\in\sub_1(Z)$; the circle found will be denoted by $\{ n \}$.
  The point of intersection of $\{n\}$ and $\{i\}$ distinct form $\emptyset$ 
  will be denoted by $\{i,n\}$ for each $i\in Z$. This way all the points of $\sub_2(X)$
  are already realized on $\naszasfera$. 
  \par
  One can say: proceeding this way we obtain the realization of each point and block
  of $\HervSpace(X)$. This is a bit hand-waving argument: let us make it more precise.
  Recall that each point $a$ and each block $a$ of $\HervSpace(X)$ are already realized
  whenever $n\notin a$.
  \begin{sentences}\itemsep-2pt\em
  \item\label{thm:reprezent:krok1}
    Assume that each point $a$ and each block $a$ of $\HervSpace(X)$ with $|a|\leq k$
    are properly realized on $\naszasfera$ (i.e. they are points and circles on 
    $\naszasfera$ that satisfy the incidences required in the definition of $\HervSpace(X)$).
    Let $n \in p \in \sub_{k+1}(X)$ be a point of $\HervSpace(X)$. Then all the circles
    $ p\setminus\{ i \} $ with $i\in p$ have a common point on $\naszasfera$.
    \begin{proof}[\proofname\ of \eqref{thm:reprezent:krok1}]
      The claim is evident for $k=2$, so we assume that $k\geq 3$.
      Let us write $a = p \setminus \{ n \}$. Take any distinct $i,j\in a$: we shall prove
      that $a$, $p\setminus\{i\}$, and $p\setminus\{j\}$ have a common point (distinct from
      $a\setminus\{ i,j \}$). From assumptions, there is $l\in a$ with $l\neq i,j$.
      Set $x = a \setminus \{i,j,l\}$ 
      (then $x$ is a point of $\HervSpace(X)$).
      Observe the schema presented in Figure \ref{diag:extend1}.
      \begin{figure}[ht]
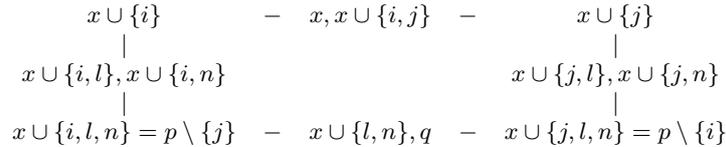

      \begin{center}\footnotesize\begin{math}
      \begin{array}{ccccc}
        x\cup\{i\}  & - & x, x\cup\{i,j\} & - & x\cup\{ j \} 
        \\
	| & \strut & \strut & \strut & |
	\\
	x\cup\{i,l\}, x \cup\{i,n\} & \strut & \strut & \strut & x\cup\{j,l\}, x\cup\{ j,n \}
	\\
	| & \strut & \strut & \strut & |
	\\
        x \cup \{ i,l,n \} = p\setminus\{ j \}  & - & x \cup\{ l,n \}, q 
        & - & x \cup\{ j,l,n \} =  p \setminus \{ i \}
      \end{array}\end{math}
      \end{center}
      \caption{Illustration to the proof of \eqref{thm:reprezent:krok1} in \ref{thm:reprezent}}
      \label{diag:extend1}
      \end{figure}
      Between the circles of the diagram there are placed their intersection points.
      From the inductive assumption of this paragraph, the points
      $x$, $x\cup\{ i,n \}$, $x\cup\{ j,n \}$, and $x\cup\{ l,n \}$ are on the circle
      $x\cup\{ n \}$
      From the Miquel Axiom we infer that $q$ lies on the circle which contains
      $x\cup\{ i,j \}$, $x\cup\{ i,l \}$, and $x\cup\{ j,l \}$; the latter circle is $a$,
      so $a$ passes through $q$, as required.
    \end{proof}
  \item\label{thm:reprezent:krok2}
    Assume that each point $a$ and each block $a$ of $\HervSpace(X)$ with $|a|\leq k$
    are properly realized on $\naszasfera$. 
    Let $n \in p \in \sub_{k+1}(X)$ be a block of $\HervSpace(X)$.
    Then all the points $ p\setminus\{ i \}$ with $i\in p$ lie on a circle on $\naszasfera$.
    \begin{proof}[\proofname\ of \eqref{thm:reprezent:krok2}]
      As above, without loss of generality it suffices to assume that $k\geq 4$.
      Let us write $a = p\setminus\{ n \}$. Let $i,j,m\in a$ be pairwise distinct;
      from assumption, there is $l\in a$ with $l\neq i,j,m$. We shall prove that
      the points $a$, $p\setminus\{i\}$, $p\setminus\{j\}$, and $p\setminus\{m\}$
      are on a circle on $\naszasfera$.
      Let $x = a \setminus\{i,j,m,l\}$ and consider the diagram in Figure \ref{diag:extend2}:
      \begin{figure}[ht]
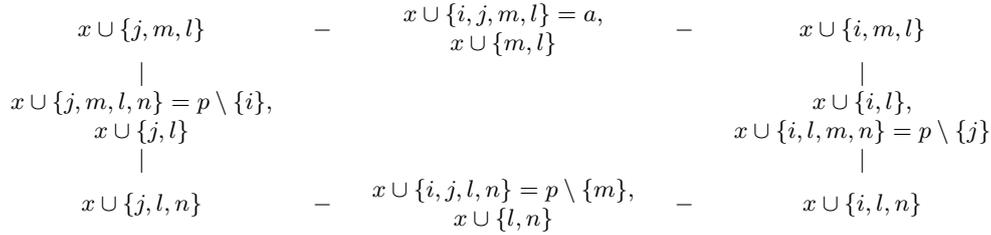

      \begin{center}\footnotesize\begin{math}
      \begin{array}{ccccc}
        x\cup\{j,m,l\}  & - & \begin{array}{c} x\cup\{i,j,m,l\} = a, \\ x\cup\{m,l\} \end{array}
        & - & x\cup\{ i,m,l \} 
	\\
	| & \strut & \strut & \strut & |
	\\
	\begin{array}{c}x\cup\{j,m,l,n\}=p\setminus\{ i \}, \\ x \cup\{j,l\} \end{array}
	& \strut & \strut & \strut & 
	\begin{array}{c} x\cup\{i,l\}, \\ x\cup\{ i,l,m,n \} = p \setminus\{ j \} \end{array}
	\\
	| & \strut & \strut & \strut & |
	\\
        x \cup \{ j,l,n \}   & - &
        \begin{array}{c} x \cup\{ i,j,l,n \} = p \setminus\{ m \}, \\ x\cup\{ l,n \} \end{array}
        & - & x \cup\{ i,l,n \}  
      \end{array}\end{math}
      \end{center}
      \caption{Illustration to the proof of \eqref{thm:reprezent:krok2} in \ref{thm:reprezent}}
      \label{diag:extend2}
      \end{figure}
      From inductive assumption, the points 
      $x\cup\{ m,l \}$,
      $x\cup\{ j,l \}$
      $x\cup\{ i,l \}$, and
      $x\cup\{ l,n \}$ are on the circle $x \cup\{ l \}$. From the Miquel Axiom we get
      the desired claim.
    \end{proof}
  \end{sentences}
  In view of the above we can inductively extend given realization from the elements of 
  $\sub_k(X)$ to the elements of $\sub_{k+1}(X)$ thus obtaining a required embedding.
\end{proof}

This, what is perhaps more interesting is an observation, which follows from 
the way in which our spherical representation was defined:
\begin{cor}
  Let $\goth M$ be a Miquelian Benz plane (i.e. let it be a sphere, a hyperboloid, or
  a cylinder in a Pappian projective 3-space, 
  with the family of conics on it distinguished) such that the 
  size of its chain is at least $n+1$.
%%  rank of  a pair of its points is at least $n$.
  The structure $\HervSpace(n)$ can be realized on $\goth M$, with the blocks interpreted
  as chains of $\goth M$.
\end{cor}
\begin{cor}
  For each positive integer $n$ the structure $\HervSpace(n)$ can be realized on a 
  hyperboloid and on a cylinder in the real projective 3-space. 
\end{cor}

Theorem \ref{thm:reprezent} for infinite countable $X$ was already remarked
in \cite{babbage}.
Note that the configuration $\HervSpace(n)$ realized on a real 2-sphere
(in more modern language: on the real M{\"o}bius plane)
is referred also as a {\em Clifford Configuration}
(cf. \cite{rigby1}, \cite{rigby2}, \cite{ziegen} and many others).
Babbage gives in \cite{babbage} for an infinite Clifford Configuration the labelling
of its points and circles by the finite sets of positive integers,
which directly shows that it is isomorphic to the Cox configuration,
no connection like this is indicated, though.

\medskip
%%  Let $A\in\sub(X)$ be an arbitrary set. 
  Let $X$ be an arbitrary set and $A\subset X$.
  Write 
  $\tau_A(Z) = A\mid Z = (A\setminus Z)\cup(Z\setminus A)$ 
  for any set $Z$.
Slightly generalizing \ref{fct:wstep} we see that
{\em if $X = X_1 \cup X_2$ is finite, $X_1\cap X_2 = \emptyset$,
and $A\in \sub(X_{3-i})$, $i\in\{1,2\}$ then the set
\begin{ctext}
  $\tau_A({\HervSpace(X_i)}) = 
  \left\{ \tau_A(a)\colon a\in\sub(X_i) \right\}$
\end{ctext}
yields a subconfiguration ${\HervSpSymb}^A$  of $\HervSpace(X)$
isomorphic to $\HervSpace(X_i)$.}  %% for $i=1,2$.}
Particularly, $\tau_\emptyset({\HervSpace(X_i)})$
is simply $\HervSpace(X_i)$ naturally embedded in $\HervSpace(X)$
(comp. \ref{fct:wstep}).
{\em The two families 
$(\HervSpSymb)^i  = \{ {\HervSpSymb}^A\colon A\in\sub(X_{3-i}) \}$
$i=1,2$
are mutually transversal and cover the underlying Cox configuration
$\HervSpace(X)$ in such a way that each its flag is a flag of one
of the configurations 
$\HervSpSymb^A$, $A \in\sub(X_1)\cup\sub(X_2)$.}
On one hand this observation allows us to decompose a Cox configuration
into smaller, simpler Cox configurations.
E.g. it enables us to imagine $\HervSpace(8)$ as a union of
$2^4+2^4 = 32$ M{\"o}bius configurations $\HervSpace(4)$,
$\HervSpace(6)$ as  a union of $2^3 + 2^3 = 16$ `spherical trianges' $\HervSpace(3)$,
or 
$\HervSpace(9)$ as a union of $2^5$ M{\"o}bius configurations and
$2^4$ configurations $\HervSpace(5)$.
On the other hand these decompositions correspond to 
decompositions of the (suitable) cube graphs (cf. \ref{fct:hypcub})
into smaller cube graphs
($Q_d$ decomposed into $Q_{d_1}, Q_{d_2},\ldots,Q_{d_t}$,
where $d = d_1+\ldots+d_t$).
%%$\HervSpace$

%%%%%%%%%%%%%%%%%%%%%%%%%%%%%%%%%%%%%%%%%%%%%%%%%%%%%%%%%%%%%%%%%%%%%%%%
%%%%%%%%%%%%%%%%%%%%%%%%%%%%%%%%%%%%%%%%%%%%%%%%%%%%%%%%%%%%%%%%%%%%%%%%
%%%%%%%%%%%%%%%%%%%%%%%%%%%%%%%%%%%%%%%%%%%%%%%%%%%%%%%%%%%%%%%%%%%%%%%% COX:AUT
%%%%%%%%%%%%%%%%%%%%%%%%%%%%%%%%%%%%%%%%%%%%%%%%%%%%%%%%%%%%%%%%%%%%%%%%
%%%%%%%%%%%%%%%%%%%%%%%%%%%%%%%%%%%%%%%%%%%%%%%%%%%%%%%%%%%%%%%%%%%%%%%%

\section{Automorphisms of Cox Configurations}\label{sec:cx:automorfy}

The following are immediate.
\begin{lem}\label{lem:aut1}
  Let $\varphi \in S_X$ and let $\widetilde{\varphi}$ be the natural extension of $\varphi$
  to a bijection of $\sub(X)$. Clearly, the map $\varphi\longmapsto\widetilde{\varphi}$
  is a group monomorphism $S_X\longrightarrow S_{\sub(X)}$.
 \par
  Then
  $\widetilde{\varphi}\in\Aut({\HervSpace(X)})$ for each $\varphi\in S_X$.
\end{lem}

\begin{lem}\label{lem:aut2}
  Let $A\in\sub(X)$ be arbitrary finite. 
%%  Write 
%%  $\tau_A(Z) = A\mid Z = (A\setminus Z)\cup(Z\setminus A)$ 
%%  for any subset $Z$ of $X$.
  Then
  $$
  \tau_A \text{ is }\left\{
  \begin{array}{ll} 
    \text{ an automorphism of } \HervSpace(X) & \text{when } 2 \mathrel{|} |A|,  
	\\
    \text{ a correlation of } \HervSpace(X) & \text{when } 2 \mathrel{\not{|}} |A|.	
  \end{array}\right.
  $$
\end{lem}
As a consequence of \ref{lem:aut2} we get an evident
\begin{cor}\label{cor:cx:homog}
  The structure $\HervSpace(X)$ is homogeneous and self-dual for any set $X$.
  Moreover, its automorphism group is flag-transitive.
  \par
  Any two stabilizers
  $({\Aut({\HervSpace(X)})})_{A}$ with $A\in\sub(X)$ are isomorphic.
\end{cor}
\begin{lem}\label{lem:aut3}
  The action of $({\Aut({\HervSpace(X)})})_{\emptyset}$ on the set
  $\sub_{<\omega}(X)$   %% of this component 
  coincides with the action of the group
  $\big\{ \widetilde{\varphi}\colon \varphi\in S_X  \big\}$ on this set. 
\end{lem}
\begin{proof}
  Let $f\in({\Aut({\HervSpace(X)})})_{\emptyset}$, then $f$ yields a permutation of the blocks
  through $\emptyset$; right from the definition this means that there is $\varphi\in S_X$
  such that
  $f(\{ i \}) = \{ \varphi(i) \} = \widetilde{\varphi}(\{ i \})$ for each $i \in X$.
  Next, $f$ maps the points of intersection of the above  blocks onto itself,
  particularly, $f(\{ i,j \}) = \{ \varphi(i),\varphi(j) \} = \widetilde{\varphi}(\{ i,j \})$
  for each pair $i,j$ of distinct elements of $X$. Continuing in this fashion
  we end up with $f(a) = \widetilde{\varphi}(a)$ for each $a\in\sub_k(X)$, 
  for every positive integer $k$.
\end{proof}

Let $\varphi\in S_X$, $A\subset X$.
Clearly,
$$ \widetilde{\varphi}\circ \tau_A \circ {\widetilde{\varphi}}^{-1} = 
  \tau_{\widetilde{\varphi}(A)}, $$

Summing up \ref{lem:aut1}, \ref{lem:aut2}, and \ref{lem:aut3} we conclude with
\begin{cor}
  The group of collineations and correlations of 
  $\HervSpace(X)$ is isomorphic to 
  the semidirect  product 
  $S_X\ltimes \struct{ \sub_{<\omega}(X), \mid }$.
  In the particular case $|X| = n <\omega$, the group of collineations and correlations
  of $\HervSpace(n)$ is thus isomorphic to the semidirect product
  $S_n \ltimes C_2^n$.
\end{cor}

Evidently, the set
$\{A\subset X \colon 2 \mathrel{|} |A| < \infty  \}$
is a subgroup of the 2-group $\struct{\sub(X),\mid}$.
This yields immediately (cf. also \cite{coxet})
\begin{cor}\label{cor:Cox:aut} 
  The group of collineations of 
  $\HervSpace(X)$ is isomorphic to 
  the semidirect  product 
  $S_X\ltimes \struct{ \{ a\in\sub_{<\omega}(X)\colon 2 \mathrel{|} |a| \}, \mid }$.
  \par
  In particular, 
  $\Aut({\HervSpace(n)}) \cong S_n \ltimes C_2^{n-1}$
\end{cor}

\begin{rem}
  Let $|X| = 2k$ and $k > 2$. Then 
  $\Aut({{\bf K}^\dagger(X,k)})\cong C_2 \oplus S_X$.
\normalfont
  This can be directly deduced from 
  known characterizations of the automorphisms of combinatorial Grassmannians,
  and of the maps which preserve binary collinearity of points in such Grassmannians.
  In view of \ref{lem:aut1} and \ref{lem:aut2} this means also that
  each automorphism of 
    ${{\bf K}^\dagger(X,k)}$ 
  can be extended to an automorphism of $\HervSpace(X)$:
  an automorphism of ${{\bf K}^\dagger(X,k)}$ is (a restriction of) a map of the form
    $(\tau_X)^\varepsilon\circ\widetilde{\varphi}$, 
    where $\varphi\in S_X$, and $\varepsilon=0,1$.
\end{rem}

A characterization of the automorphisms of geometrical Cox and Clifford
configurations is not so evident: an automorphism of a configuration 
$\goth D$ realized in another geometrical structure $\goth M$ is commonly 
required to be a ``geometrical" automorphims i.e. 
it is required to be an {\em automorphism of} $\goth M$ which leaves $\goth D$ invariant.
Classification of such automorphisms of $\goth D$
may depend on the way in which $\goth D$ is embedded into $\goth M$
(comp. e.g. \cite{coxdes} in the case of automorphims of the Desargues Configuration
on the projective plane).
In any case, a geometrical automorphism of $\goth D$ is an (ordinary)
automorphism of $\goth D$ as well. Let us have a look at the following example.
\begin{exm}
  Let ${\goth D} = \HervSpace(n)$ be realized on a M{\"o}bius plane $\goth M$
  and $\varphi\in S_n$.
  A geometrical automorphism $\widetilde{\varphi}$ (if it exists) permutes circles
  through a fixed point.
  In particular, for every circle $\{ i \}$ the harmonic ratios
  $(\emptyset,\{ i,j_1 \}//\{ i,j_2 \},\{ i,j_3 \})$
  and
  $(\emptyset,\{ \varphi(i),\varphi(j_1) \}//
   \{ \varphi(i),\varphi(j_2) \},\{ \varphi(i),\varphi(j_3) \})$
  should coincide for each three-element set
  $\{ j_1,j_2,j_3\} \subset \{ 1,...,n \} \setminus \{ i \}$ .
  Thus, e.g., the permutation 
  $(1)(2)(3,4)(5)\in\Aut({\HervSpace(5)})$
  does not extend to any geometrical automorphism.
  Slightly more generally:
  {\em If $\id\neq\varphi\in S_X$, $|\Fix(\varphi)|\geq 3$ then $\widetilde{\varphi}$
  is never a geometrical automorphism of $\goth D$}.
  \par
  A systematic study of the admissible sizes of the groups of geometrical 
  automorphisms of the Cox (Clifford) configurations is addressed to other papers.
\myend
\end{exm}

%%%%%%%%%%%%%%%%%%%%%%%%%%%%%%%%%%%%%%%%%%%%%%%%%%%%%%%%%%%%%%%%%%%%%%%%
%%%%%%%%%%%%%%%%%%%%%%%%%%%%%%%%%%%%%%%%%%%%%%%%%%%%%%%%%%%%%%%%%%%%%%%%
%%%%%%%%%%%%%%%%%%%%%%%%%%%%%%%%%%%%%%%%%%%%%%%%%%%%%%%%%%%%%%%%%%%%%%%% BIBLIO
%%%%%%%%%%%%%%%%%%%%%%%%%%%%%%%%%%%%%%%%%%%%%%%%%%%%%%%%%%%%%%%%%%%%%%%%
%%%%%%%%%%%%%%%%%%%%%%%%%%%%%%%%%%%%%%%%%%%%%%%%%%%%%%%%%%%%%%%%%%%%%%%%

\bigskip\bigskip
\noindent
Authors' address:
\\
Ma{\l}gorzata Pra{\.z}mowska, Krzysztof Pra{\.z}mowski
\\
Institute of Mathematics
\\
University of Bia{\l}ystok
\\
\strut\quad ul. Akademicka 2
\\
\strut\quad 15-267 Bia{\l}ystok, Poland
\\
e-mail: {\ttfamily malgpraz@math.uwb.edu.pl},
{\ttfamily krzypraz@math.uwb.edu.pl}

\end{document}